\documentclass[12pt,reqno]{amsart}

\usepackage{AKstyle,comment}

\numberwithin{equation}{section}

\begin{document}

\title[Whittaker Plancherel measure on GL($n$)]{On the Determination of the Plancherel measure for Lebedev-Whittaker transforms on GL($n$)}
\author{Dorian Goldfeld}
\email{goldfeld@math.columbia.edu}
\address{Columbia University, New York, NY 10027}
\thanks{D.G. is partially supported by  NSF grant DMS-1001036}
\author{Alex Kontorovich}
\email{alexk@math.sunysb.edu}
\address{Stony Brook University, Stony Brook, NY 11794}
\thanks{A.K. is partially supported by NSF grants DMS-1064214 and DMS-1001252.}
\date{\today}
\dedicatory{Dedicated to Andrzej Schinzel on the occasion of his 75th birthday}
\maketitle

\section{Introduction}

The classical Plancherel formula states that the inner product of two functions is the same as the inner product of their (Fourier) transforms.
This fact has been vastly generalized, and the measure appearing on the transform side is called the Plancherel measure.
One of Harish-Chandra's great achievements was the determination of the Plancherel measure for reductive Lie groups
%An
% rigorous 
%exposition of Harish-Chandra's construction of Plancherel measure for real groups is given in
(see e.g. Wallach  \cite%[\S15.10.3]
{Wallach1992}).
%The recent work of Sakellaridis-Venkatesh \cite{SakellaridisVenkatesh2010} gives a Plancherel formula for spherical varieties of split groups over $p$-adic fields.
%, a program for developing a Plancherel formula for spherical varieties of split groups over $p$-adic fields is developed.
%\\

One type of Plancherel measure for real groups comes from the Lebedev-Whittaker transform, the earliest version of which is
the Kontorovich-Lebedev transform, 
see  \cite{KontorovichLebedev1938, KontorovichLebedev1939}.
The original transform, 
a type of index transform involving modified Bessel functions, was introduced
to solve certain boundary-value problems. 
It has been useful in many applications in modern analytic number theory (see e.g. \cite{IwaniecKowalski2004}), since
it can be characterized  as a Whittaker transform on $\GL(2).$ As such it has a natural generalization to reductive Lie groups.
%A generalization of the Kontorovich-Lebedev transform for real reductive groups 
%(see \cite[\S15.10.3]{Wallach1992}). 
%\\

The main aim of this paper is to obtain a very concrete and explicit version of the
Plancherel measure for the
 %Kontorovich-
Lebedev-Whittaker transform for the real group $\GL(n, \R)$ with $n \ge 2.$ We expect  that such a realization will be useful 
for analytic methods for number theory on higher rank groups.
%in the future. 
\\

For $n \ge 2$, consider an admissible irreducible cuspidal automorphic representation  $\pi$ for GL$(n,\mathbb A)$, where $\mathbb A$ is the adele group over $\Q$.  By Flath's tensor product  theorem \cite{Flath1979}, 
$$
\pi = \otimes \pi_v,
$$ 
where the tensor product goes over irreducible, admissible, unitary local representations of $GL(n, \Q_p).$ We shall assume that $\pi$ is unramified at infinity.

To characterize the real components of such representations in a more explicit manner,  we introduce,
for $n \ge 2$,  the generalized upper half plane
 $$
\frak h^n := GL(n, \R)/\left(O(n, \R) \cdot \R^\times\right)
.
$$ 
Let $\frak D^n$ denote the algebra of $GL(n, \R)$--invariant differential operators acting on $\frak h^n$. By the Iwasawa decomposition, every $z \in \frak h^n$ may be uniquely written in the form $z = xy$ with $x \in U_n(\R)$ (the group of unipotent upper triangular matrices in $GL(n, \R)$) and $y$ a diagonal matrix  of the form
 \begin{equation} y = \begin{pmatrix} y_1\cdots y_{n-1} & & &\\
& \ddots & &\\
& & y_1 &\\
& & & 1 \end{pmatrix}, \qquad  \big(y_i > 0 \; \text{for} \; i = 1, \ldots,n-1\big). \label{y}\end{equation}
Whenever we write $z = xy \in \frak h^n$ we assume that $x,y$ are as described above.

 Let $W_\infty$ be a Whittaker model for $\pi_\infty$. Then there exists a {\it spherical Whittaker function} $W \in W_\infty$ which is $K_\infty$--fixed for the maximal compact subgroup  $K_\infty = O(n, \Bbb R).$ Then $W:\frak h^n \to \Bbb C$ is characterized by the fact that $W$ is an eigenfunction of  $\frak D^n$, and in addition,
$$W(uz) = \psi(u) \cdot W(z), \qquad (z \in \frak h^n),$$
for any $u \in U_n(\R)$ and some fixed character $\psi$ of $U_n(\R).$ Associated to $\pi_\infty$, there will exist spectral parameters $\nu =(\nu_1, \ldots \nu_{n-1}) \in \C^{n-1}$  so that we may write (see \cite[\S5.9]{Goldfeld2006} for the {\it completed} Jacquet-Whittaker function, which is used exclusively throughout this paper) 
\begin{equation} 
W(z) =
 \prod_{j=1}^{n-1}
\prod_{j\le k\le n-1}
\pi^{-\foh-v_{j,k}}
\G\left(\foh + v_{j,k}\right)
\cdot
%\left(
 \int_{U_n(\R)} I_{\nu}(w_n u z) \psi(u) \, d^\times u
%\right)
, \qquad (z   \in \frak h^n)
.
\label{WhittakerFunction}
\end{equation}
Here $\G$ is the Gamma function, $w_n$ is the long element of the Weyl group, the $I$-function is given by 
$$%\begin{equation}
I_\nu(z) = \prod_{i=1}^{n-1} \prod_{j=1}^{n-1} y_i^{b_{ij}\cdot\nu_j}, \qquad \left(z = xy \in \frak h^n\right),
$$%\end{equation}
with
\begin{equation}
b_{ij} = \begin{cases} ij, & \text{if $i+j \le n,$} \\
(n-i)(n-j), & \text{if $i + j \ge n,$} \end{cases}
\label{eq:bijDef}
\end{equation} 
and
$$
v_{j,k}=
\sum_{i=0}^{j-1}
{n\nu_{n-k+i}-1\over2}
.
$$
There should be no confusion between the real number $\pi=3.14\dots$ in \eqref{WhittakerFunction}, and the representation $\pi$.

Since we assumed that the local representation $\pi_\infty$ is tempered, it follows that 
\be\label{eq:nuTot}
\nu_j = 1/n + it_j
\ee 
with $t_j \in \R$  $\;(j = 1, \ldots,n-1).$ In this case we define $W_{it} := W$ where $W$ is given by \eqref{WhittakerFunction} and $t = (t_1, \ldots, t_{n-1})$.
The Haar measure on the Levy component is given by
$$
d^\times y = \prod\limits_{k=1}^{n-1} y_k^{-k(n-k)} \, \frac{dy_k}{y_k}.
$$ 

\begin{definition}[Lebedev-Whittaker transform]\label{def:1.1}  {\it Let $f:\R_+^{n-1} \to \C$, let $y$ be as in \eqref{y}, and let $t = (t_1, \ldots, t_{n-1}) \in \R^{n-1}$. 
Then we define the Lebedev-Whittaker transform  $f^\#:\R_+^{n-1} \to \C$ by
$$f^\#(t) := \int_{\R_+^{n-1}} f(y) W_{it}(y) \; d^\times y,$$
provided the above integral converges absolutely.  }
\end{definition} 

  The  inverse  transform  is given in the next theorem. 
Let $\alpha=(\ga_{1},\dots,\ga_{n})\in\R^{n}$ be linear functions of $t\in\R^{n-1}$ defined  as follows. Recall $b_{k\ell}$ defined in \eqref{eq:bijDef}. For $1 \le k \le n-1$, the $\alpha_k$ are determined by (see (5.9.7) in \cite{Goldfeld2006})
 \be\label{eq:gaDefN}
 \frac{k(r-k)}{2} \; + \; \sum_{\ell=1}^{r-k} \frac{\alpha_\ell}{2} \; = \; \sum_{\ell=1}^{n-1} b_{k \ell}\cdot \nu_\ell
 \ee
  and
  $$\alpha_n = - \sum_{k=1}^{n-1} \alpha_k.$$   
  \begin{theorem}[Lebedev-Whittaker inversion] 
  \label{thm:main}
  Let $f:\R_+^{n-1}\to\C$ be smooth of compact support, and $f^\#:\R_+^{n-1}\to\C$ be given as in Definition \ref{def:1.1}. Then
  $$
  f(y) = 
{1\over \pi^{n-1}}
  \int_{\R^{n-1}} f^\#(t) W_{-it}(y)
  \frac{dt}{\prod_{1\le k \ne \ell\le n} \Gamma\left(\frac{\alpha_k - \alpha_\ell}{2}  \right)}  %dt
.
  $$ 
 \end{theorem} 

\begin{rmk}
The above inversion formula is not new \cite{Wallach1992}, but the novelty of 
  our approach is its explicit presentation and derivation.
  Our proof uses only complex analysis (the residue theorem) and the location of poles and residues of the Gamma function.
Admittedly, it relies crucially on Stade's \cite{Stade2002} formula (see \S\ref{recall} below), but this is again a vast generalization of Barnes' Lemma.  We expect our methods to have other applications in higher rank analytic number theory. 
\end{rmk}
\begin{rmk}
As we are mainly interested in the Plancherel measure,
we restrict our attention  to smooth functions of compact support. Of course the inversion holds for a much wider class of test functions.
\end{rmk}
  
  As a consequence, we have the Plancherel formula for the Lebedev-Whittaker transform.
\begin{cor}
For $f_{1},f_{2}:\R_{+}^{n-1}\to\C$, smooth of compact support, we have
\bea
\label{KLPGLn}
\<f_1,f_2\>
&=&
\int_{\R_{+}^{n-1}}
f_1(y)\, 
\overline{f_2(y)}
\ 
d^{\times}y
%{dy_1dy_2\over y_1^3y_2^3}
\\
\nonumber
=
\<f_1^\sharp,f_2^\sharp\>
&=&
{1\over \pi^{n-1}} 
\int_{\R^{n-1}}
{h_1^\sharp(t)\, 
\overline{
h_2^\sharp(t) 
}
}\
  \frac{dt}{\prod_{1\le k \ne \ell\le n} \Gamma\left(\frac{\alpha_k - \alpha_\ell}{2}  \right)}  %dt
.
\eea
\end{cor}

Thus the measure
$$
  \frac{dt}{\prod_{1\le k \ne \ell\le n} \Gamma\left(\frac{\alpha_k - \alpha_\ell}{2}  \right)}  %dt
$$
is the Plancherel measure for the Lebedev-Whittaker transform on $\GL(n,\R)$. Notice that by taking the product of half of the Gamma functions in the  denominator, i.e. by taking  
$
\prod_{1\le k < \ell\le n} \Gamma\left(\frac{\alpha_k - \alpha_\ell}{2}  \right),
$
we obtain the Harish-Chandra $c$-function, $c(i\nu)$ (see Wallach \cite[\S15.10.3]{Wallach1992}), so the measure can also be written, after the linear change of variables \eqref{eq:nuTot}, as
$$
\frac{d\nu}{c(i\nu)c(-i\nu)}
.
$$

\subsection*{Organization}
This paper is organized as follows. In \S\ref{recall}, we recall %Jacquet's Whittaker function and 
Stade's formula, which is a key ingredient in our proof.  
%Since a rigorous exposition of Theorem \ref{}
A sketch of the
 proof of the Lebedev-Whittaker inversion formula is given in \S 3.
For ease of notation we restrict to the case $n=3$, that is, $\GL(3)$; it will be clear how to proceed on $\GL(n)$. 
%In the Appendix, we give an elementary proof of the Mellin inversion  theorem.
As an after thought, we also treat in the appendix the case of $GL(1)$,  by giving an elementary proof of the Mellin inversion formula.

\subsection*{Acknowledgements}
The authors would like to thank Valentin Blomer for comments and corrections to an earlier draft.
\\

%%%%%%%%%%%%%%%%%%%%%%%%%%%%%%%%%%%%
\section{%Jacquet-Whittaker Functions and 
Stade's Formula}\label{recall}

We use the notation setup in the previous section.  
 Recall that we are assuming that $\pi_\infty$ is unramified, which implies that the eigenvalue parameters $\nu$ are tempered, i.e. $\nu_j=1/n+ i t_j$ with $t_j\in \R$, see \eqref{eq:nuTot}. Let $\mu_{j}=1/n+iu_{j}$ with $u_{j}\in\R$, and define $\gb_{j}$ related to $u_{j}$ in the same way as $\ga_{j}$ are related $t_{j}$, that is \eqref{eq:gaDefN}. 
% 
%We begin with
%Recall
 Stade's formula  for $\GL(n)$ (see \cite[Prop 11.6.17]{Goldfeld2006}) is as follows. 
\begin{theorem}[\cite{Stade2002}]
Let $n \ge 2.$ Then for $t,u\in\R^{n-1}$,  $s\in\C$ with $\Re(s)\ge1$,
\bea\label{stadeGn}
&&
\int_{\R_{+}^{n-1}}%{y_1=0}^\infty\cdots\int_{y_{n-1}=0}^\infty 
W_{it}(y) W_{iu}(y) \prod_{j=1}^{n-1}y_j^{(n-j)s}d^{\times}y%{dy_j\over y_j^{1+j(n-j)}}
\\
\nonumber
&&
\hskip1in
=
%(-1)^{{n(n+1)+2\over 2}}{2^{(n+1)n(n-1)\over 6}  \pi^{-s{n(n-1)\over 2}} } 
{1\over 2\,\pi^{s\frac{n(n-1)}2}}
\; \dfrac{1}{  \Gamma\left({ns\over2}\right)}\prod\limits_{j=1}^n\prod\limits_{k=1}^n \Gamma\left({s+\ga_j+\gb_k\over 2}\right) .
\eea
\end{theorem}

\

%%%%%%%%%%%%%%%%%%%%%%%%%%%
%%%%%%%%%%%%%%%%%%%%%%%%%%%%%%%%

\section{%Proof for 
Lebdev-Whittaker inversion for GL$(3)$
}

We now specialize to $n=3$ for ease of exposition. In this case, 
%Define for $\GL(3,\R)$ 
the Lebedev-Whittaker transform of a smooth, compactly supported test function $h:\R_{+}^{2}\to\C$ becomes
\be\label{eq:KLt}
h^\sharp(t_{1},t_{2}):=
\int_{y_{1}=0}^{\infty} 
\int_{y_{2}=0}^{\infty} 
h(y_{1},y_{2}) 
W_{it_{1},it_{2}}
(y_{1},y_{2})
{
dy_{1}dy_{2}
\over
y_{1}^{3}
y_{2}^{3}
}
.
\ee

Note that $h^{\sharp}(t_{1},t_{2})$ inherits the same functional equations as $W_{it_{1},it_{2}}$; these are invariance under permutation of the parameters $\ga_{1},\ga_{2},\ga_{3}$, defined by (cf. \eqref{eq:gaDefN}) 
\be\label{eq:gaDef}
\ga_{1}=2it_{1}+it_{2},\qquad 
\ga_{2}=-it_{1}+it_{2},\qquad
\text{ and }\qquad
\ga_{3}=-it_{1}-2it_{2}.
\ee
%see (5.9.7) in \cite{Goldfeld2006}.

The inverse transform is given as follows.
%
%Define space of test functions.....
%
For 
$H:\R^{2}\to\C%\in\cH_{3}^{\sharp}
$ having the above symmetries in $(t_{1},t_{2})$,  let
\bea\label{eq:KLiDef}
H^{\flat}(y_{1},y_{2})
&=&
{1\over \pi^2} 
\int\limits_{t_{1}=-\infty}^{\infty} 
\int\limits_{t_{2}=-\infty}^{\infty}  
{
H
(t_1,t_2) 
W_{-it_1,-it_2}
(y_1,y_2) 
}
%\\
%\nonumber
%\\
%\nonumber
%&&
%\qquad
%\times
{
dt_1dt_2
\over 
\prod_{1\le \ell\neq \ell' \le 3} \G\left({\ga_{\ell}-\ga_{\ell'}\over 2}\right)
}
,
\eea
assuming the integral converges absolutely. %, and  
%where the $\ga_{\ell}$ are defined in \eqref{eq:gaDef}.

The following is a restatement of Theorem \ref{thm:main}.

\begin{thm}[Lebedev-Whittaker Inversion]\label{thm:KLT}
$$
(H^{\flat})^{\sharp}=H.
$$
\end{thm}

\begin{comment}

As a corollary of the above, we have determined the Plancherel measure for Whittaker transforms on $\GL(n)$.
\begin{cor}
For $h_{1},h_{2}:\R_{+}^{2}\to\C$, smooth of compact support, we have
\bea
\label{KLP}
\<h_1,h_2\>
&=&
\int_0^\infty \int_0^\infty h_1(y_1,y_2)\, 
\overline{h_2(y_1,y_2)}
\ 
{dy_1dy_2\over y_1^3y_2^3}
\\
\nonumber
=
\<h_1^\sharp,h_2^\sharp\>
&=&
{1\over \pi^2} \int_\R \int_\R  {h_1^\sharp(t_1,t_2)\, 
\overline{
h_2^\sharp(t_1,t_2) 
}
}\
{
dt_1dt_2
\over 
%\prod_{1\le j\neq k\le3}
\left|
\G\left(
{
3it_1
 \over 2}
\right)
\right|^2
\left|
\G\left(
{
3it_2
 \over 2}
\right)
\right|^2
\left|
\G\left(
{
3it_1
+
3it_2
 \over 2}
\right)
\right|^2
}
.
\eea
\end{cor}
%Notice that all of the functions involved are actually real-valued if the $h$.
\end{comment}

%We now sketch a proof of Theorem \ref{thm:KLT}, part {\bf (2)}. 

\pf[Sketch of the proof% of Theorem \ref{thm:KLT}
]

%Fix $y=(y_{1},y_{2})$,  $t=(t_{1},t_{2})$ and $\ga_j$ as above. The dummy variables are $u=(u_1,u_{2})$ and $\gb_j$, related to $u_{j}$ in the same way as $\ga_{j}$ to $t_{j}$.

Assume that 
 the function 
$H(t_{1},t_{2})$
is invariant under permutations of $\ga_{1},\ga_{2},\ga_{3}$,
 that the integral \eqref{eq:KLiDef} converges absolutely,
and that for some $\vep_{0}>0$,
 $H(t_{1},t_{2})$ is  holomorphic in the region 
$$
|\Im(t_{1})|,|\Im(t_{2})|<2\vep_{0}
.
$$ 
For any $0<\vep<\vep_{0}$, define
\bea
\nonumber
\cH(t_{1},t_{2},\vep)
&:=&
\int_{y_{1}=0}^{\infty}
\int_{y_{2}=0}^{\infty}
 H^{\flat}(y_{1},y_{2})\
 W_{-it_{1},-it_{2}}
 %(y)
\left(
y_{1},y_{2}
%\matthree {y_{1}y_{2}}{}{}{}{y_{1}}{}{}{}1
\right)
 (y_1^2y_2)^\vep 
 \,
 {dy_1\over y_1^3}{dy_2\over y_2^3}.
 \eea
 Then clearly the limit of the above as $\vep\to0$ converges to $(H^{\flat})^{\sharp}$. Hence 
we must show that $\cH(t_{1},t_{2},\vep)\to H(t_{1},t_{2})$ as $\vep\to0$. For simplicity, we assume that the $\ga_{j}$ are all distinct.

 Insert the definition of $H^{\flat}$:
 \beann
\cH(t_{1},t_{2},\vep)
&=&
\int_{y_{1}=0}^{\infty}
\int_{y_{2}=0}^{\infty}
\left(
%%%%%%%%%%%%%%%%%%%%%%%%%
%%%%%%%%%%%%%%%%%%%%%%%%%
{1\over \pi^2} 
\int\limits_{t'_{1}=-\infty}^{\infty} 
\int\limits_{t'_{2}=-\infty}^{\infty}  
{
H
(t'_1,t'_2) 
W_{-it'_1,-it'_2}
(y_{1},y_{2}) 
%\left(
%\matthree {y_{1}y_{2}}{}{}{}{y_{1}}{}{}{}1
%\right)
}
\right.
%\\
%\nonumber
\\
&&
\qquad\qquad
\left.
\times
{
dt'_1dt'_2
\over 
%\prod_{1\le j\neq k\le3}
\G\left(
{
3it'_1
 \over 2}
\right)
\G\left(
{
-3it'_1
 \over 2}
\right)
%\right|^2
%\left|
\G\left(
{
3it'_2
 \over 2}
\right)
\G\left(
{
-3it'_2
 \over 2}
\right)
%\right|^2
%\left|
\G\left(
{
3it'_1
+
3it'_2
 \over 2}
\right)
\G\left(
{
-3it'_1
-
3it'_2
 \over 2}
\right)
%%%%%%%%%%%%%%%
\begin{comment}
\left|
\G\left(
{
3it'_1
 \over 2}
\right)
\right|^2
\left|
\G\left(
{
3it'_2
 \over 2}
\right)
\right|^2
\left|
\G\left(
{
3it'_1
+
3it'_2
 \over 2}
\right)
\right|^2
\end{comment}
%%%%%%%%%%%%%%
}
%%%%%%%%%%%%%%%%%%%
%%%%%%%%%%%%%%%%%%%%%
\right)
 \\
&&
\hskip3.5in
\times
 W_{-it_{1},-it_{2}}
 (y)
%\left(
%\matthree {y_{1}y_{2}}{}{}{}{y_{1}}{}{}{}1
%\right)
 (y_1^2y_2)^\vep 
 {dy_1\over y_1^3}{dy_2\over y_2^3}
.
 \eeann
Interchanging orders, one inserts Stade's Formula \eqref{stadeGn} with $s=\vep$. 
Simplifying gives:
\beann
\hskip-10pt
\cH(t_{1},t_{2},\vep)
&=&
{1\over \pi^2} 
\int\limits_{t'_{1}=-\infty}^{\infty} 
\int\limits_{t'_{2}=-\infty}^{\infty}  
H
(t'_1,t'_2) \
\Bigg[
{1
\over
{
 \pi^{3\vep}
\G(3\vep/2)
}
2
}
\
\G\left({\vep-2it_{1}-it_{2}+2it'_{1}+it'_{2}\over2}\right) 
\\
&&
\quad\times
\G\left({\vep-2it_{1}-it_{2}-it'_{1}+it'_{2}\over2}\right) 
\G\left({\vep-2it_{1}-it_{2}-it'_{1}-2it'_{2}\over2}\right) 
\\
&&
\quad\quad
\times
\G\left({\vep+it_{1}-it_{2}+2it'_{1}+it'_{2}\over2}\right) 
\G\left({\vep+it_{1}-it_{2}-it'_{1}+it'_{2}\over2}\right) 
\\
&&
\quad\quad \quad
\times
\G\left({\vep+it_{1}-it_{2}-it'_{1}-2it'_{2}\over2}\right) 
\G\left({\vep+it_{1}+2it_{2}+2it'_{1}+it'_{2}\over2}\right) 
\\
&&
\quad\quad \quad \quad
\times
\G\left({\vep+it_{1}+2it_{2}-it'_{1}+it'_{2}\over2}\right) 
\G\left({\vep+it_{1}+2it_{2}-it'_{1}-2it'_{2}\over2}\right) 
\Bigg]
 \\
 &&
 \hskip60pt
\times
{
dt'_1dt'_2
\over 
%\prod_{1\le j\neq k\le3}
%\left|
\G\left(
{
3it'_1
 \over 2}
\right)
\G\left(
{
-3it'_1
 \over 2}
\right)
%\right|^2
%\left|
\G\left(
{
3it'_2
 \over 2}
\right)
\G\left(
{
-3it'_2
 \over 2}
\right)
%\right|^2
%\left|
\G\left(
{
3it'_1
+
3it'_2
 \over 2}
\right)
\G\left(
{
-3it'_1
-
3it'_2
 \over 2}
\right)
%\right|^2
}
.
\eeann

We make the change of variables $(t_{1}',t_{2}')\mapsto(\ga_{1}',\ga_{2}')$, where
(see \eqref{eq:gaDef}) 
$$
\ga_{1}' = 2it_{1}' + it_{2}',\qquad \text{and} \qquad
\ga_{2}' = -it_{1}' + it_{2}'
.
$$
The Jacobian is $|\det (\dd\ga'/\dd t')|=-3$. Similarly, we use the notation \eqref{eq:gaDef} to simplify the appearance of the above expression, which is now:

\beann
\hskip-20pt
\cH(t_{1},t_{2},\vep)
&=&
{1\over \pi^2} 
\int\limits_{\ga'_{1}=-i\infty}^{i\infty} 
\int\limits_{\ga'_{2}=-i\infty}^{i\infty}  
H
\left(
{\ga'_1-\ga'_{2}\over 3i},
{\ga'_1+2\ga_{2}'\over 3i}
\right) \
\\
&&
\hskip1in
\times
\Bigg[
{1
\over
{
 \pi^{3\vep}
\G(3\vep/2)
}
2
}
\
\G\left({\vep-\ga_{1}+\ga_{1}'\over2}\right) 
\G\left({\vep-\ga_{1}+\ga_{2}'\over2}\right) 
\G\left({\vep-\ga_{1}+\ga_{3}'\over2}\right) 
\\
&&
\hskip1.1in
\times
\G\left({\vep-\ga_{2}+\ga_{1}'\over2}\right) 
\G\left({\vep-\ga_{2}+\ga_{2}'\over2}\right) 
\G\left({\vep-\ga_{2}+\ga_{3}'\over2}\right) 
 \\
 &&
\hskip1.2in
\times
\G\left({\vep-\ga_{3}+\ga_{1}'\over2}\right) 
\G\left({\vep-\ga_{3}+\ga_{2}'\over2}\right) 
\G\left({\vep-\ga_{3}+\ga_{3}'\over2}\right) 
\Bigg]
\\
&&
\hskip66pt
\times
{
{-1\over3}
d\ga'_1d\ga'_2
\over 
%\prod_{1\le j\neq k\le3}
%\left|
\G\left(
{
2\ga'_1+\ga'_{2}
 \over 2}
\right)
\G\left(
{
-2\ga'_1-\ga'_{2}
 \over 2}
\right)
%\right|^2
%\left|
\G\left(
{
\ga'_{1}-\ga'_{2}
 \over 2}
\right)
\G\left(
{
-\ga'_{1}+\ga'_{2}
 \over 2}
\right)
%\right|^2
%\left|
\G\left(
{
\ga'_{1}+2\ga'_{2}
 \over 2}
\right)
\G\left(
{
-\ga'_{1}-2\ga'_{2}
 \over 2}
\right)
%\right|^2
}
.
\eeann

%%%%%%%%%%%%%%%%%%%%%%%%%%%%%%%

 Shift line of integration from $\ga'_{1}\in\{i\R\}$ to $\ga'_{1}\in\{-\vep_{0}+i\R\}$, with $\vep_{0}>\vep$. We pass through poles at
 \beann
 \ga'_{1} &=& -\vep+\ga_{1},\qquad\qquad\text{with residue $\cR_{1}$},
\\ 
\ga'_{1} &=& -\vep+\ga_{2},\qquad\qquad\text{with residue $\cR_{2}$},
\\ 
\ga'_{1} &=& -\vep-\ga_{1}-\ga_{2},\qquad\text{with residue $\cR_{3}$}.
 \eeann

%%%%%%%%%%%%%%%%%%%%%%%%%%%%%

Consider $\cR_{1}$. After some cancellations, we have
\beann
\cR_{1}
&=&
{i\over\pi}\int_{\ga'_{2}=-i\infty}^{i\infty}
{
H\left(
{-\vep+\ga_{1}-\ga'_{2}\over3i},{-\vep+\ga_{1}+2\ga'_{2}\over3i}
\right)
\over
\pi^{2\vep}\G(3\vep/2)2
}
\Bigg[
%\G\left(
%{\vep-\ga_{1}+\ga'_{2}\over2}
%\right)
%\G\left(
%{2\vep-2\ga_{1}-\ga'_{2}\over2}
%\right)
\G\left(
{\ga_{1}-\ga_{2}\over2}
\right)
\G\left(
{\vep-\ga_{2}+\ga'_{2}\over2}
\right)
\\
%%%%%%%%%%%
&&
\hskip10pt
\times
\G\left(
{2\vep-\ga_{1}-\ga_{2}+\ga'_{2}\over2}
\right)
\G\left(
{2\ga_{1}+\ga_{2}\over2}
\right)
\G\left(
{\vep+\ga_{1}+\ga_{2}+\ga'_{2}\over2}
\right)
\G\left(
{2\vep+\ga_{2}-\ga'_{2}\over2}
\right)
\Bigg]
\\
%%%%%%%%%%%%%%%
&&
\hskip20pt
\times
{
{-1\over3}
d\ga'_2
\over 
\G\left(
{-2\vep+2\ga_{1}+\ga'_{2} \over 2}
\right)
%\G\left(
%{2\vep-2\ga_{1}-\ga'_{2} \over 2}
%\right)
\G\left(
{-\vep+\ga_{1}-\ga'_{2} \over 2}
\right)
%\G\left(
%{\vep-\ga_{1}+\ga'_{2} \over 2}
%\right)
\G\left(
{-\vep+\ga_{1}+2\ga'_{2}\over 2}
\right)
\G\left(
{\vep-\ga_{1}-2\ga'_{2} \over 2}
\right)
}
\eeann

%%%%%%%%%%%%%%%%%%%%%%%%%%%%
Next, in the $\cR_{1}$ integral, we shift the line of integration to the left, from $\ga'_{2}\in\{i\R\}$ to $\ga'_{2}\in\{-\vep_{0}+i\R\}$. Now there are poles at:
 \beann
 \ga'_{2} &=& -\vep+\ga_{2},\qquad\qquad\qquad\text{with residue $\cR_{1,1}$},
\\ 
 \ga'_{2} &=& -\vep-\ga_{1}-\ga_{2},\qquad\qquad\text{with residue $\cR_{1,2}$.}
 \eeann
In total there are six such residues $\cR_{j,k}$, $1\le j \le 3$, $1\le k\le2$. In fact, by the invariance of $H$ in permutations of $\ga_{1},\ga_{2},\ga_{3}$, these residues all have the same  contribution. We now evaluate $\cR_{1,1}$. After simplification, we have
\beann
\cR_{1,1}
&=&
{1\over 2\cdot3\cdot \pi^{3\vep}}
H\left(
{\ga_{1}-\ga_{2}\over 3i},{-3\vep+\ga_{1}+2\ga_{2}\over3i}
\right)
\Bigg[
\G\left(
{2\ga_{1}-\ga_{2}\over2}
\right)
\G\left(
{\ga_{1}+2\ga_{2}\over2}
\right)
\Bigg]
\\
&&
\hskip160pt
\times
{1
\over
\G\left(
{-3\vep+2\ga_{1}-\ga_{2}\over2}
\right)
\G\left(
{-3\vep+\ga_{1}+2\ga_{2}\over2}
\right)
}
\\
&\to&
\frac16
H(t_{1},t_{2})
,
\eeann
as $\vep\to0$.
Hence the contribution from the six residues adds up to exactly $H(t_{1},t_{2})$.
The remaining integrals all contain the factor $\G\left(\frac{3\vep}2\right)$ in the denominator, making the integrals vanish as $\vep\to0$.
This completes the proof, under the assumption that the $\ga_{j}$ are all distinct.

Had the $\ga_{j}$ not been distinct, we would have had poles of order two in the contour shifting argument; the rest of the analysis is similar.
%This completes the proof. 
\epf
\
%%%%%%%%%%%%%%%%%%%%%%%%%%%%%%%%
%%%%%%%%%%%%%%%%%%%%%%%%%%%%%%%%

\appendix

\section{An elementary proof of Mellin inversion}

Fix some smooth, compactly supported test function $f:\R^{+}\to\C$. 
Define the Mellin transform 
\be\label{mdef}
\tilde f(s):=\int_0^\infty f(y)y^s {dy\over y}
\ee
and the Mellin inverse transform
\be\label{imdef}
h(x):={1\over 2\pi i}\int_{(2)} \tilde f(s) x^{-s}ds.
\ee
%These are equivalent to the Fourier transform and inverse after a change of variables, e.g. in \eqref{gdef} set $u=\log y$.
\begin{theorem}[Mellin inversion]
 $f(x)=h(x)$.
\end{theorem}
\begin{proof}
We require the well-known formula
\be\label{elt}
{1\over 2\pi i }\int_{(2)} x^s {ds \over s(s+1)} = \twocase{}{1-\frac1x,}{if $x>1$}{0,}{if $x<1$.}
\ee
Starting with \eqref{mdef}, integrate by parts twice:
\beann
\tilde f(s)
&=&
-
\int_0^\infty f'(y){y^{s}\over s} {dy}
\\
&=&
\int_0^\infty f''(y){y^{s+1}\over s(s+1)} {dy}
.
\eeann
Insert this into \eqref{imdef}, reverse orders of integration and apply \eqref{elt}:
\beann
h(x)
&=&
{1\over 2\pi i}\int_{(2)} 
\left(
\int_0^\infty f''(y){y^{s+1}\over s(s+1)} {dy}\,
\right)
x^{-s}ds
\\
&=&
\int_0^\infty   f''(y)
\left(
{1\over 2\pi i}\int_{(2)}
\left(\frac yx\right)^{s}  {ds\over s(s+1)} 
\right)
 \,y\,{dy}
\\
&=&
\int_x^\infty   f''(y)
\left(1-\frac xy\right)
 \,y\,{dy}.
\eeann
And now integrate by parts twice (in the reverse direction):
\beann
h(x)
&=&
\int_x^\infty   f''(y)
\left(y-x\right)
 \,{dy}
 \\
&=&
-
\int_x^\infty   f'(y)
\left(1\right)
 \,{dy}
 \\
&=&
f(x)
 .
\eeann
\vskip-.2in
\end{proof}

Hence the Plancherel formula for the Mellin transform reads
$$
\int_{0}^{\infty}f_{1}(y)\overline{f_{2}(y)}
=
\<f_{1},f_{2}\>
=
\<\tilde f_{1},\tilde f_{2}\>
=
{1\over 2\pi i}
\int_{(0)}
\tilde f_{1}(s)
\overline{\tilde f_{2}(s)}
ds,
$$
that is, the Plancherel measure is just Lebesgue measure.

\bibliographystyle{alpha}

\bibliography{AKbibliog}

\end{document}